\documentclass{amsart}
\usepackage{amsfonts,amsmath,amstext,amssymb,amsthm, mathrsfs,amscd,enumerate,verbatim}
\usepackage[all]{xy}

\newtheorem{thm}{Theorem}[section]

\newtheorem{exa}[thm]{Example}
\newtheorem{lem}[thm]{Lemma}

\newtheorem{rem}[thm]{Remark}

\theoremstyle{definition}

\numberwithin{equation}{section}
\newcommand{\Hom}[3]{\operatorname{Hom}_{#1}\left(#2,#3\right)}

\newcommand{\ext}[4]{\operatorname{Ext}^{#1}_{#2}\left(#3,#4\right)}

\newcommand{\cd}[2]{\operatorname{cd}_{R}\left({#1}, #2\right)}
\newcommand{\Ht}[2]{\operatorname{ht}_{#1} {(#2)}}

\newcommand{\h}[3]{\operatorname{H}^{#1}_{#2}\left(#3\right)}
\newcommand{\gam}[2]{\Gamma_{#1}\left(#2\right)}

\newcommand{\ann}{\operatorname{Ann}}

\newcommand\supp {\operatorname{Supp}}
\newcommand\spec{\operatorname{Spec}}
\newcommand\ass {\operatorname{Ass}}
\newcommand\mass {\operatorname{MinAss}}
\newcommand\assh {\operatorname{Assh}}

\newcommand{\grad}[2]{\operatorname{grade}\left(#1,#2\right)}

\newcommand\depth{\operatorname{depth}}

\newcommand{\id}{\operatorname{inj\,dim}}
\newcommand\im{\operatorname{im}}

\newcommand\fP{\mathfrak P}

\newcommand\fa{\mathfrak a}
\newcommand\fb{\mathfrak b}
\newcommand\fc{\mathfrak c}
\newcommand\fm{\mathfrak m}
\newcommand\fn{\mathfrak n}
\newcommand\fp{\mathfrak p}
\newcommand\fq{\mathfrak q}

\newcommand\N{\mathbb N}
\swapnumbers

\begin{document}

\title[Annihilators of local cohomology and Ext]{Some bounds for the annihilators of local cohomology and Ext modules}%
\author[A.~fathi]{Ali Fathi}
\address{Department of Mathematics, Zanjan Branch,
Islamic Azad University,  Zanjan, Iran.}
\email{fathi\_ali@iauz.ac.ir,  alif1387@gmail.com}
\begin{abstract}
Let $\mathfrak a$ be an ideal of a commutative Noetherian ring $R$ and $t$ be a non-negative integer. Let $M$ and $N$  be two finitely generated $R$-modules.  In certain cases, we give some bounds under inclusion for the annihilators of $\operatorname{Ext}^t_R(M, N)$ and
 $\operatorname{H}^t_{\mathfrak a}(M)$   in terms of  minimal primary decomposition of the zero submodule of $M$ which are independent of the choice of minimal  primary decomposition. Then, by using those bounds,  we compute the annihilators of local cohomology and Ext modules in certain cases.
 \end{abstract}
\maketitle
\setcounter{section}{-1}
\section{\bf Introduction}
Throughout the paper, $R$ is a commutative Noetherian ring with nonzero identity.
 The $i$-th local cohomology of an $R$-module $M$ with respect to an ideal $\fa$ was defined by Grothendieck as follows:
 $$\h i{\fa}M={\underset{n}{\varinjlim}\operatorname{Ext}^i_R\left(R/\fa^n,  M\right)};$$
see \cite{bs, bh, h} for   more details.

 In this section, we assume   $M$ is a  non-zero finitely generated $R$-module, $N$ is a  Gorenstein $R$-module,  $0=M_1\cap\hdots\cap M_n$ is a minimal primary decomposition of the zero submodule of $M$ with $\ass_R(M/M_i)=\{\fp_i\}$ for all $1\leq i\leq n$ and $\fa$ is an ideal of $R$. We refer the reader to \cite[Sec. 6]{mat} for basic properties of primary decomposition of modules and to \cite{s, s2} for more details about the Gorenstein modules (see also the paragraph before Lemma \ref{gor}).

  We denote, for an $R$-module $M$, $\sup\{i\in\N_0: \h i{\fa}M\neq 0\}$ by $\cd {\fa}M$. Assume $d=\dim_R(M)<\infty$. By Grothendieck's Vanishing  Theorem, $\cd{\fa}M\leq d$. When $\cd{\fa}M=d$, then we have
\begin{equation}\label{eq1}
  \ann_R\left(\h d{\fa}M\right)=\ann_R\left(M/\bigcap_{\cd{\fa}{R/\fp_i}=d}M_i\right).
\end{equation}
This equality is proved by   Lynch \cite[Theorem 2.4]{l}  whenever $R$ is a complete local ring and $M=R$.
  In \cite[Theorem 2.6]{bag}  Bahmanpour {\it et al.}  proved that $\ann_R\left(\h d{\fa}M\right)=\ann_R(M/T_R(\fa, M))$ whenever $\fa=\fm$ and $R$ is a complete local ring, where $T_R(\fa, M)$ denotes the largest submodule $N$ of $M$ such that $\cd \fa N<\cd \fa M$. Then   Bahmanpour
   \cite[Theorem 3.2]{ba} extended  the result of Lynch for the $R$-module $M$. Next, Atazadeh {\it et al.}  \cite{asn1} proved this equality whenever $R$ is a local ring (not necessarily complete) and finally  in \cite{asn2} they extended  it to the non-local case.
   (Note that $T_R(\fa, M)=\bigcap_{\cd{\fa}{R/\fp_i}=\cd {\fa}M}M_i$ \cite[Remark 2.5]{asn1}, also, if $(R, \fm)$ is a complete local ring and $\fp\in\ass_R(M)$, then, by the Lichtenbaum-Hartshorne Vanishing Theorem,   $\cd {\fa}{R/\fp}=d$ if and only if  $\dim_R(R/\fp)=d$ and $\sqrt{\fa+\fp}=\fm$).

In the first section (see Theorem \ref{annh1} and Remark \ref{rem1}), for an arbitrary integer $t$, we give a bound for the annihilator of $\ext tRMN$ in terms of minimal primary decomposition of the zero submodule of $M$. More precisely,   we show that
\begin{equation}\label{eq2}
\ann_R\left(M/\bigcap_{\fp_i\in\Delta(t)}M_i\right)\subseteq \ann_R\left(\ext tRMN\right)\subseteq\ann_R\left(M/\bigcap_{\fp_i\in\Sigma(t)}M_i\right)
\end{equation}
where $\Delta(t)=\{\fp\in\ass_R(M)\cap\supp_R(N): \Ht R\fp\leq t\}$, $\Sigma(t)=\{\fp\in\mass_R(M)\cap\supp_R(N): \Ht R\fp=t\}$ and $\mass_R(M)$ denotes the set of minimal elements of $\ass_R(M)$.
If $t=\grad {\ann_R(M)}N<\infty$, then the above index sets are equal and we can compute the annihilator of $\ext tRMN$. Note that, in general, for an arbitrary integer $t$, there is not a subset $\Sigma$ of $\ass_R(M)$ such that $ \ann_R\left(\ext tRMN\right)=\ann_R\left(M/\bigcap_{\fp_i\in\Sigma}M_i\right)$; see Example \ref{exa1}.

In the second section, we consider the annihilators of local cohomology modules. By using the above bound on the annihilators of Ext modules, when $(R, \fm)$ is a local ring, we show, in Theorem \ref{annh}, that \begin{equation}\label{eq3}
\ann_R\left(M/\bigcap_{\fp_i\in\Delta'(t)}M_i\right)\subseteq \ann_R\left(\h t{\fm}M\right)\subseteq\ann_R\left(M/\bigcap_{\fp_i\in\Sigma'(t)}M_i\right),
\end{equation}
where $\Delta'(t)=\{\fp\in\ass_R(M): \dim_R(R/\fp)\geq t\}$  and $\Sigma'(t)=\{\fp\in\mass_R(M): \dim_R(R/\fp)=t\}$. Next, whenever $R$ is not necessarily local, in Theorem \ref{annh2}, we give a bound for  the  annihilator of the top local cohomology module $\h {\cd{\fa}M}{\fa}M$   which implies  equality (\ref{eq1}) when $d=\cd{\fa}M$. Finally, for each $t$, in Theorem \ref{annh3}, we provide a bound for the annihilator of $\h t{\fa}M$ when $M$ is Cohen-Macaulay, and also we compute its annihilator at $t=\grad{\fa}M$. All the given bounds are independent of the choice of minimal primary decomposition. We adopt the convention that the intersection of empty family of submodules of an $R$-module $M$ is $M$.

   \section{\bf Bounds for the annihilators of  Ext-modules}
Assume  $M, N$ are  finitely generated $R$-modules such that $N$ is a Gorenstein module, and $0=M_1\cap\hdots\cap M_n$ is a minimal primary decomposition of the zero submodule of $M$ with $\ass_R(M/M_i)=\{\fp_i\}$ for all $1\leq i\leq n$. We refer  the reader to
 \cite[Sec. 6]{mat} for basic properties and unexplained terminologies  about the primary decomposition of modules and to \cite{s,s2} for more details about  the Gorenstein modules. In this section  [Theorem \ref{annh1}], for each integer $t$, we give a  bound for the annihilator of $\ext tRMN$ in terms of minimal primary decomposition of the zero submodule of $M$ which   is independent of the choice of minimal primary decomposition. As an application, in  the case where $t=\grad {\ann_R(M)}N$, we compute the  annihilator of $\ext t RMN$. More precisely, for $t=\grad {\ann_R(M)}N$, we have
  $$\ann_R\left(\ext t RMN\right)= \ann_R\left(M/\bigcap_{\fp_i\in\Sigma(t)} M_i\right)$$
  where $\Sigma(t) =\{\fp\in\mass_R(M)\cap\supp_R(N): \Ht R\fp=t\}$; see Theorem \ref{annh1} and Remark \ref{rem1}.
  Note that, in general, for an arbitrary integer $t$, there is not a subset $\Sigma$ of $\ass_R(M)$ such that $\ann_R\left(\ext t RMN\right)= \ann_R\left(M/\bigcap_{\fp_i\in\Sigma} M_i\right)$; see Example \ref{exa1}. These results will be used in the second section to compute the annihilators of local cohomology modules.

Before proving these results, we need some lemmas.

\begin{lem}[{\cite[Theorem 6.8]{mat}}]\label{ass} Let $M$ be a non-zero finitely generated $R$-module. Let
 $\ass_R(M)=\{\fp_1,\dots,\fp_n\}$, and $0=M_1\cap\hdots\cap M_n$ be a minimal primary decomposition of the zero submodule of $M$ with $\ass_R(M/M_i)=\{\fp_i\}$ for all $1\leq i\leq n$. Assume  $\Phi$ is a subset of $\ass_R(M)$ and $N=\bigcap_{\fp_i\in\Phi}M_i$.  Then $$\ass_R(M/N)=\Phi \textrm{, and } \ass_R(N)=\ass_R(M)\setminus\Phi.$$
\end{lem}

  Assume  $N$ is a  submodule of  an $R$-module $M$. For any multiplicatively  closed subset $S$ of $R$,  we denote  the contraction of $S^{-1}N$ under the canonical map $M\rightarrow S^{-1}M$ by $S_M(N)$. Assume $\Sigma\subseteq\ass_R(M)$. We say that $\Sigma$ is an isolated subset of $\ass_R(M)$ if it satisfies the following condition: if $\fq\in \ass_R(M)$ and $\fq\subseteq\fp$ for some $\fp\in\Sigma$, then $\fq\in\Sigma$.

The following lemma is  well-known, but  we  prove it  for the readers' convenience.
\begin{lem}[See {\cite[Theorem 4.10 and Exercise 4.23]{at}}]\label{primary} Let $M$ be a finitely generated $R$-module, and  $N$  a proper submodule of $M$. Let   $N=\bigcap_{i=1}^nN_i$  be a minimal  primary decomposition of $N$ in $M$ with $\ass_R(M/N_i)=\fp_i$ for all $1\leq i\leq n$. Assume $\Sigma$ is an  isolated subset of $\ass_R(M/N)$. Then
$$\bigcap_{\fp_i\in\Sigma}N_i=S_M(N),$$
where  $S=R\setminus \bigcup_{\fp\in\Sigma}\fp$.
 In particular, $\bigcap_{\fp_i\in\Sigma}N_i$ is independent of the choice of minimal primary decomposition of $N$ in $M$.
\end{lem}
\begin{proof}
Assume $\Sigma\subseteq\ass_R(M/N)$ is an isolated subset of  $\ass_R(M/N)$ and $S=R\setminus\bigcup_{\fp\in\Sigma}\fp$.
 If $S^{-1}\left({M}/{\bigcap_{\fp_i\in\ass_R(M/N)\setminus\Sigma }N_i}\right)\neq 0$, then there exists
 $$\fq\in \ass_R\left({M}/{\bigcap_{\fp_i\in\ass_R(M/N)\setminus\Sigma }N_i}\right)=\ass_R(M/N)\setminus\Sigma$$
  such that $\fq\cap S=\emptyset$. Since $\fq\cap S=\emptyset$, by the Prime Avoidance Theorem,  $\fq\subseteq\fp$ for some $\fp\in\Sigma$. But $\Sigma$ is an isolated subset of $\ass_R(M/N)$ and  so $\fq\in \Sigma$,
which is a contradiction. Hence
  $S^{-1}\left(\bigcap_{\fp_i\in\ass_R(M/N)\setminus\Sigma }N_i\right)=S^{-1}M$. It follows that  $S^{-1}N=\bigcap_{\fp_i\in\Sigma}S^{-1}N_i$. Contracting both sides under the canonical map $M\rightarrow S^{-1}M$ we obtain
  $(S^{-1}N)^c=\bigcap_{\fp_i\in\Sigma}(S^{-1}N_i)^c$. Now, assume $\fp_i\in\Sigma$. It is clear that $N_i\subseteq(S^{-1}N_i)^c$. Conversely, if $m\in (S^{-1}N_i)^c$, then $m/1\sim n/s$ for some $n\in N_i$ and  $s\in S $. Hence $tsm=tn\in N_i$ for some $t\in S$. Since $N_i$ is a $\fp_i$-primary submodule of $M$ and $ts\notin\fp_i$,  we have $m\in N_i$. Therefore  $N_i=(S^{-1}N_i)^c$, and hence $(S^{-1}N)^c=\bigcap_{\fp_i\in\Sigma}N_i$. This completes the proof.
\end{proof}

\begin{rem}\label{rem} Let the situation and notations be as in above lemma. Assume, in addition, that $\Sigma=\emptyset$,  and we  consider the above lemma in this special case separately.  It is clear that $\Sigma$ is an isolated subset of $\ass_R(M/N)$ and $\bigcap_{\fp_i\in\Sigma}N_i=M$ because the intersection of the empty family of subsets of a  set $M$ is $M$. On the other hand, we have
$S=R\setminus\bigcup_{\fp\in\Sigma}\fp=R$. Since $0\in S$, we obtain $S^{-1}(N)=S^{-1}(M)=0$, and so the contraction of  $S^{-1}(N)$ under the map $M\rightarrow S^{-1}(M)$ is $M$. Therefore  we have $S_M(N)=M=\bigcap_{\fp_i\in\Sigma}N_i$ in this case.
\end{rem}

Let $(R, \fm)$ be a local ring. A non-zero finitely generated $R$-module $G$ is said to be Gorenstein if  $$\depth_R(G)=\dim_R(G)=\id_R(G)=\depth_R(R)=\dim_R(R)$$
 (so $R$ is Cohen-Macaulay)  or equivalently  $\ext iR{R/\fm}G$ is non-zero only at $i=\dim_R(G)$; see \cite[Theorem 3.11]{s}.  More generally, if $R$ is not necessarily local,  a non-zero finitely generated $R$-module $G$ is said to be  Gorenstein if  $G_\fp$ is a Gorenstein $R_\fp$-module for all $\fp\in\supp_R(G)$; see \cite[corollary 3.7]{s}.
 When $(R, \fm)$ is a complete local ring, then  Gorenstein modules under  isomorphism are the  non-empty finite direct sums of the canonical module  \cite[Corollary 2.7]{s2}.

 The following property  of Gorenstein modules is needed in the proof the main theorem of this section.
\begin{lem}\label{gor} Let $G$ be a Gorenstein $R$-module, and  $\fp$ a  prime ideal of $R$. Then
 $\fp\in\supp_R(G)$ if and only if $G\neq\fp G$.
\end{lem}
\begin{proof}
Assume $\fp\in\supp_R(G)$. Hence $G_\fp\neq 0$ and consequently $G_\fp\neq\fp R_{\fp}G_\fp$ by  Nakayama's Lemma. It follows that $G\neq\fp G$.
Conversely, assume $G\neq\fp G$. Thus there exists  $\fq\in\supp_R(G)$ such that $G_\fq\neq \fp R_q G_\fq$. Therefore $\fp \subseteq\fq$, and hence \cite[Corollary 4.14]{s} implies that  $\fp\in\supp_R(G)$.
\end{proof}
Now we are ready to state and prove the main theorem of this section which provides   bound for the annihilators of Ext modules. Local version of this theorem  Remark \ref{rem1}(1.1) will be used to compute the annihilators of local cohomology modules in the next section.
  \begin{thm}\label{annh1} Let  $M, N$ be non-zero finitely generated $R$-modules,
  and let $0=M_1\cap\hdots\cap M_n$ be  a minimal primary decomposition of the zero submodule of $M$ with $\ass_R(M/M_i)=\{\fp_i\}$ for all $1\leq i\leq n$. Let $t\in\N_0$ and set $\Delta(t)=\{ \fp\in\ass_R(M): \grad{\fp}N\leq t\}$, $\Sigma(t)=\{ \fp\in\mass_R(M): \grad{\fp}N=t\}$, $S^t=R\setminus\bigcup_{\fp\in\Delta(t)}\fp$, and $T^t=R\setminus\bigcup_{\fp\in\Sigma(t)}\fp$.
Then
\begin{enumerate}[\rm(i)]
\item  $\bigcap_{\fp_i\in\Delta(t)} M_i=S^t_M(0)$ and $\bigcap_{\fp_i\in\Sigma(t)} M_i=T^t_M(0)$. In particular, $\bigcap_{\fp_i\in\Delta(t)} M_i$ and $\bigcap_{\fp_i\in\Sigma(t)} M_i$ are independent of the choice of minimal primary decomposition of the zero submodule of $M$.
    \item $S^t_M(0)$ is the largest submodule $L$ of $M$ such that $\ext iRLN=0$ for all $i\leq t$.
 \item  $$\ann_R({M}/{S^t_M(0)})\subseteq\ann_R\left(\ext t RMN\right).$$
  If, in addition,    $N$ is a Gorenstein module, then
  $$\ann_R\left(\ext t RMN\right)\subseteq \ann_R({M}/{T^t_M(0)}).$$
  \item If $N$ is a Gorenstein module such that $\supp_R(M)\cap\supp_R(N)\neq\emptyset$ and  $t=\grad{\ann_R(M)}N$, then $\Delta(t)=\Sigma(t)$ and
  $$\ann_R\left(\ext t RMN\right)= \ann_R({M}/{T^t_M(0)}).$$
\end{enumerate}
\end{thm}
\begin{proof} Set $S=S^t_M(0)$ and $T=T^t_M(0)$.

i) Since $\Delta(t)$ and $\Sigma(t)$ are isolated subsets of $\ass_R(M)$, (i) is an immediate consequence of Lemma \ref{primary}.

ii) By Lemma \ref{ass}, in view of  \cite[Proposition 1.2.10]{bh}, we have
  \begin{align*}
  \grad{\ann_R(S)}N&=\grad{\sqrt{\ann_R(S)}}N=\grad{\bigcap_{\fp\in\ass_R(S)}\fp}N\\
&=\min_{\fp\in\ass_R(S)}\grad{\fp}N=\min_{\fp\in\ass_R(M)\setminus\Delta(t)}\grad{\fp}N>t.
  \end{align*}
  Since $\grad{\ann_R(S)}N>t$, we have  $\ext iR{S}N=0$ for all $i\leq t$ by \cite[Proposition 1.2.10(e)]{bh}. Also, we note that, if $\Delta(t)=\ass_R(M)$, then $S=0$ and $\grad{\ann_R(S)}N=\grad RN=\infty$. Now, assume $L$ is a submodule of $M$ such that  $\ext iRLN=0$ for all $i\leq t$. Suppose, for the sake of contradiction, that $L\nsubseteq S$. Then
$$0\neq L/(L\cap S)\cong (L+S)/S\subseteq M/S.$$
Thus $\emptyset\neq\ass_R(L/(L\cap S))\subseteq \Delta(t)$. Hence, there exists $\fp\in\ass_R(L/(L\cap S))\subseteq V(\ann_R(L))$ such that $\grad {\fp}N\leq t$. But this is impossible,    because, by our assumption,  $\grad{\ann_R(L)}N>t$; see again \cite[Proposition 1.2.10(e)]{bh}. Hence $L\subseteq S$ and the proof of (ii) is completed.

 iii) Since $\ext tRSN=0$, the exact sequence $0\rightarrow S\rightarrow M\rightarrow M/S\rightarrow 0$ induces the epimorphism
$\ext tR{M/S}N\rightarrow\ext tRMN.$
 It follows that $$\ann_R(M/S)\subseteq\ann_R\left(\ext tR{M/S}N\right)\subseteq\ann_R\left(\ext tR{M}N\right)$$
  and hence the first inclusion in (iii) holds. To prove the second inclusion in (iii), assume that $N$ is a Gorenstein module. If $\Sigma(t)=\emptyset$, then $T=M$ by Remark \ref{rem} and there is nothing to prove. Hence,   suppose that $\Sigma(t)\neq\emptyset$,  $\fp_i\in\Sigma(t)$ and $y\in\ann_R\left(\ext tRMN\right)$. Since $\grad{\fp_i}N=t<\infty$, we have  $\fp_iN\neq N$ and so, by Lemma \ref{gor}, $\fp_i\in\supp_R(N)$.  Hence  $N_{\fp_i}$ is a  Gorenstein ${R_{\fp_i}}$-module \cite[Corollary 3.7]{s}. Because $N$ is Cohen-Macaulay,   we have  $\dim_{R_{\fp_i}}( N_{\fp_i})=\grad {\fp_i}N=t$ and so, by \cite[Theorem 4.12]{s}, we have  $\dim_{R_{\fp_i}}(R_{\fp_i})=\dim_{R_{\fp_i}} (N_{\fp_i})=t$. We proved that  $N_{\fp_i}$ is a Gorenstein $R_{\fp_i}$-module of dimension $t$, and hence, in view of the faithfully flatness of completion,  we  can deduce that  $\widehat{N_{\fp_i}}$ is also a  Gorenstein $\widehat{R_{\fp_i}}$-module of dimension $t$. Hence $\widehat{N_{\fp_i}}\cong\omega^{\alpha}_{\widehat{R_{\fp_i}}}$ for some $\alpha\in\N$ \cite[Corollary 2.7]{s2}, where $\omega_{\widehat{R_{\fp_i}}}$ denotes the canonical module of $\widehat{R_{\fp_i}}$. Since $\widehat{R_{\fp_i}}$ is a Cohen-Macaulay complete local ring of dimension $t$, by the Local Duality Theorem  \cite[Theorem 11.2.8]{bs}  and \cite[Remarks 10.2.2(ii)]{bs}, we have
  \begin{align*}
\ann_{R_{\fp_i}}\left(\ext t{R_{\fp_i}}{ M_{\fp_i}}{N_{\fp_i}}\right)
&=R_{\fp_i}\cap\ann_{\widehat{R_{\fp_i}}}\left(\ext t{\widehat{R_{\fp_i}}}{\widehat{ M_{\fp_i}}}{\widehat{N_{\fp_i}}}\right)\\
&=R_{\fp_i}\cap\ann_{\widehat{R_{\fp_i}}}\left(\ext t{\widehat{R_{\fp_i}}}{\widehat{ M_{\fp_i}}}{\omega_{\widehat{R_{\fp_i}}}}\right)\\
&=R_{\fp_i}\cap\ann_{\widehat {R_{\fp_i}}}\left(\gam {\widehat{\fp_i{R_{\fp_i}}}}{\widehat{M_{\fp_i}}}\right)\\
&=\ann_{R_{\fp_i}}\left(\gam {\fp_i{R_{\fp_i}}}{M_{\fp_i}}\right)=\ann_{R_{\fp_i}}{(M_{\fp_i})}
\end{align*}
(note that since $\fp_i$ is a minimal element of $\ass_R(M)$, it follows that $\dim_{R_{\fp_i}}(M_{\fp_i})=0$ and hence
 $\gam {\fp_i{R_{\fp_i}}}{M_{\fp_i}}={M_{\fp_i}}$).

 Now, if $1\leq j\neq i\leq n $, then $(M/M_j)_{\fp_i}=0$, because  $\ass_R(M/M_j)=\{\fp_j\}$ and $\fp_i$ is a minimal element of $\ass_R(M)$. Thus
 $(M_j)_{\fp_i}= M_{\fp_i}$ for all $1\leq j\neq i\leq n$, and so $\left(\bigcap_{j=1}^nM_j\right)_{\fp_i}\cong(M_i)_{\fp_i}$. Since  $M_{\fp_i}\cong (M/0)_{\fp_i}\cong \left(M/{\bigcap_{j=1}^nM_j}\right)_{\fp_i}\cong(M/M_i)_{\fp_i}$, we have $y/1\in(\ann_R(M/M_i))_{\fp_i}$, and hence $y/1\sim z/s$ for some  $z\in\ann_R(M/M_i)$, $s\in R\setminus\fp_i$. Thus $rsy=rz\in \ann_R(M/M_i)$ for some
  $r\in R\setminus\fp_i$. Hence $rsyM\subseteq M_i$. Since $M_i$ is a $\fp_i$-primary submodule of $M$, it follows from  $rs\notin\fp_i$ that $yM\subseteq M_i$.
  Because $\fp_i$ is an arbitrary element of $\Sigma(t)$, $yM\subseteq \bigcap_{\fp_i\in\Sigma(t)}M_i$, and by part (i), this implies that $yM\subseteq T$.
Thus   $\ann_R\left(\ext tRMN\right)\subseteq\ann_R(M/T)$.

iv) Assume  $N$ is a Gorenstein module such that $\supp_R(M)\cap\supp_R(N)\neq\emptyset$. Thus $N/(\ann_R(M))N\neq 0$,  and so $t=\grad{\ann_R(M)}N<\infty$; see \cite[Definition 1.2.6]{bh}. It is clear that $\Sigma(t)\subseteq\Delta(t)$. To prove the reverse inclusion, let $\fp\in\Delta(t)$. Since $\ann_R(M)\subseteq \fp$, we obtain $\grad{\ann_R(M)}N\leq\grad{\fp}N$ and consequently $\grad{\ann_R(M)}N=\grad{\fp}N$.  Now, let $\fq\in\supp_R(M)$ be such that $\fq\subseteq\fp$.  It follows from $\grad {\fp}N=t<\infty$ that $\fp\in\supp_R(N)$,  and so $\fq\in\supp_R(N)$ by \cite[Corollary 4.14]{s}. Hence, by \cite[Theorem 2.1.3 (b)]{bh} and \cite[Theorem 4.12]{s}, we have
\begin{align*}
t&=\grad{\ann_R(M)}N\leq\grad{\fq}N=\dim_{R_\fq}(N_{\fq})\\
&=\dim_{(R_\fp)_{\fq R_\fp}}(N_\fp)_{\fq R_\fp}=\dim_{R_\fp}(N_{\fp})-\dim_{R_\fp}(N_{\fp}/(\fq R_{\fp})N_\fp)\\
&=\grad{\fp}N-\dim_{R_\fp}(R_{\fp}/\fq R_{\fp})=t-\dim_{R_\fp}(R_{\fp}/\fq R_{\fp})
\end{align*}
Therefore $\dim_{R_\fp}(R_{\fp}/\fq R_{\fp})=0$ or equivalently $\fq=\fp$. Hence $\fp\in\mass_R(M)$ and consequently  $\Delta(t)\subseteq\Sigma(t)$.
\end{proof}
For an integer $t$ and an $R$-module $M$, we denote, respectively, the sets $\{\fp\in\ass_R(M): \dim_R (R/\fp)=t\}$ and $\{\fp\in\ass_R(M): \dim_R (R/\fp)\geq t\}$ by $\ass_R^t(M)$ and $\ass_R^{\geqslant t}(M)$.
 Similarly, $\mass_R^t(M)$ and $\mass_R^{\geqslant t}(M)$ are defined as above by replacing $\ass_R(M)$ by $\mass_R(M)$. Also, when $\dim_R(M)<\infty$, the set of prime ideals in $\ass_R(M)$ of the highest possible dimension  $\{\fp\in\ass_R(M): \dim_R (R/\fp)=\dim_R(M)\}$ is denoted by $\assh_R(M)$.
\begin{rem}\label{rem1} Let the situation and notations be as in above theorem. Let $N$ be a Gorenstein $R$-module, and $\fp$  a prime ideal of $R$. Then $\grad {\fp}N=t<\infty$ if and only if  $N\neq\fp N$ or equivalently $\fp\in\supp_R(N)$. Also, if $\fp\in\supp_R(N)$, then $N_\fp$ is a Gorenstein module on the local ring $R_\fp$ and, in view of \cite[Theorem 4.12]{s}, we have
$$\grad{\fp}N=\dim_{R_\fp}(N_{\fp})=\dim_{R_\fp}(R_{\fp})=\Ht R\fp.$$
  Hence $$\Delta(t)=\{\fp\in\ass_R(M)\cap\supp_R(N): \Ht R\fp\leq t\},$$
   $$\Sigma(t)=\{\fp\in\mass_R(M)\cap\supp_R(N): \Ht R\fp=t\}.$$
 In the remainder of this remark,  assume in addition that $R$ is a local ring of dimension $d$. Then $\Ht R\fp=d-\dim_R(R/\fp)$ and $\supp_R(N)=\spec (R)$. Thus  above theorem states that
  \begin{align}\label{ext1}
\ann_R\left({M}/{\bigcap_{\fp_i\in\ass_R^{\geqslant d-t}(M)}M_i}\right)&\subseteq\ann_R\left(\ext tRMN\right)\\
 \nonumber &\subseteq \ann_R\left({M}/{\bigcap_{\mass_R^{d-t}(M)}M_i}\right).
 \end{align}
  In particular, if $M\neq 0$, then $$\grad{\ann_R(M)}N=\dim_R(N)-\dim_R(N/(\ann_R(M))N)=d-\dim_R(M)$$ and    the equality in Theorem \ref{annh1}(iv) can be rewrite as follows
  \begin{equation}\label{ext2}
    \ann_R\left(\ext {d-\dim_R(M)}RMN\right)=\ann_R\left({M}/{\bigcap_{\fp_i\in\assh_R(M)}M_i}\right).
  \end{equation}
\end{rem}

  These results are needed  in the proof of the main theorem of the  next section (Theorem \ref{annh}) which provides  some bounds for the  annihilators of local cohomology modules.

  We end this section by two examples to show that how we  can compute the above bounds for the annihilators  of Ext modules.  Moreover,  these examples show that to improve the upper bound in (1.1) we can not replace the index set $\mass_R^{d-t}(M)$ by the larger  sets $\mass_R^{\geqslant d-t}(M)$, $\ass_R^{d-t}(M)$ or $\ass_R^{\geqslant d-t}(M)$ and also to improve the lower bound in (1.1) we can not replace the index set $\ass_R^{\geqslant d-t}(M)$ by the smaller set $\ass_R^{d-t}(M)$. Also, in general for an arbitrary integer $t$, there is not  a subset $\Sigma$ of $\ass_R(M)$ such that
$\ann_R\left(\ext tRMN\right)=\ann_R\left(M/(\bigcap_{\fp_i\in\Sigma}M_i)\right)$.

Let $U$ be a subset of an $R$-module $M$. We use $\langle U\rangle$ to denote the submodule of $M$ generated  by $U$. If $U=\{m_1,\dots, m_n\}$, then we show $\langle U\rangle$ by $\langle m_1,\dots, m_n\rangle$.
\begin{exa} \label{exa1} Let $K$ be a field and  let $R=K[[X, Y]]$ be the ring of formal power series over  $K$ in indeterminates  $X, Y$.

 Set $M=R/\langle X^2, XY\rangle$, $M_1=\langle X\rangle/\langle X^2, XY\rangle$, and $M_2=\langle X^2, Y \rangle/\langle X^2, XY \rangle$. Then $0=M_1\cap M_2$ is a minimal primary decomposition of the zero submodule of $M$ with $\ass_R(M/M_1)=\{\fp_1=\langle X\rangle\}$ and $\ass_R(M/M_2)=\{\fp_2=\langle X, Y\rangle\}$. So $\ass_R(M)=\{\fp_1, \fp_2\}$ and $\mass_R(M)=\{\fp_1\}$.
Hence, we have
$$\ass_R^{\geqslant 2-t}(M)=\left\{\begin{array}{ll}
\emptyset&{\rm if }\ t=0,\\
\{\fp_1\}&{\rm if }\ t=1,\\
\{\fp_1,\fp_2\}&{\rm if }\ t=2
\end{array}\right.$$
and
 $$\mass_R^{ 2-t}(M)=\left\{\begin{array}{ll}
\emptyset&{\rm if }\ t=0, 2,\\
\{\fp_1\}&{\rm if }\ t=1.
\end{array}\right.$$
 It follows that
 $$\ann_R\left({M}/{\bigcap_{\fp_i\in\ass_R^{\geqslant 2-t}(M)}M_i}\right)=\left\{\begin{array}{ll}
R&{\rm if }\ t=0,\\
\langle X\rangle&{\rm if }\ t=1,\\
\langle X^2, XY \rangle&{\rm if }\ t=2
\end{array}\right.$$ and
$$ \ann_R\left({M}/{\bigcap_{\fp_i\in\mass_R^{2-t}(M)}M_i}\right)=\left\{\begin{array}{ll}
R&{\rm if }\ t=0, 2,\\
\langle X\rangle&{\rm if }\ t=1.
\end{array}\right.$$
Hence,   Remark \ref{rem1} implies that   $$\Hom RMR=0,\ \ann_R\left(\ext 1RMR\right)=\langle X\rangle$$ and $$\langle X^2, XY\rangle\subseteq\ann_R\left(\ext 2RMR\right)\subseteq R.$$
 Also, since $\id_R(R)=2$, we deduce that $\ext tRMR=0$ for all $t>2$.

Now, we directly compute  $\ann_R\left(\ext tRMR\right)$ for all $t$ (specially for $t=2$). It is straightforward to see that
$$\textbf{P}: 0\longrightarrow R\stackrel{d_2}{\longrightarrow} R^2\stackrel{d_1}{\longrightarrow} R\stackrel{\epsilon}{\longrightarrow} M\longrightarrow 0$$ with $\epsilon(f)=f+\langle X^2, XY\rangle,\, d_1(f, g)=X^2f+XYg,\, d_2(f)=(Yf, -Xf)$ for all $f, g\in R$ is a projective resolution of $M$. Applying the functor $\Hom R{\cdot}R$ to the delated projective resolution $\textbf{P}_M$, we obtain the following commutative diagram
\begin{displaymath} \xymatrix{
0\ar[r]&\Hom RRR\ar[d]^{\alpha}_{\cong}\ar[r]^{d_1^*}&\Hom R{R^2}R\ar[d]^{\beta}_{\cong}\ar[r]^{d_2^*}&\Hom RRR\ar[d]^{\gamma}_{\cong}\ar[r]&0\\
0\ar[r]&R\ar[r]^{\delta_1} &R^2\ar[r]^{\delta_2} &R\ar[r] &0,
} \end{displaymath}
where $\alpha, \beta, \gamma$ are natural isomorphisms, $\delta_1(f)=(X^2f, XYf)$, and $ \delta_2(f,g)=Yf-Xg$ for all $f, g\in R$.   Hence
$$ \ext 1RMR\cong{\ker\delta_2}/{\im\delta_1}=\langle(X, Y)\rangle/\langle(X^2, XY)\rangle\cong R/\langle X\rangle,$$ and
 $$ \ext 2RMR\cong {R}/{\langle X, Y\rangle} \ \text{and} \ \ext tRMR=0 \ \text{for all}\ t\neq 1, 2$$
 (note that by our notation $\ker\delta_2$ and $\im\delta_1$ are cyclic $R$-modules  generated by the elements $(X, Y)$ and $(X^2, XY)$ of $R^2$ respectively).
 It follows that $$\ann_R \left(\ext 1RMR\right)=\langle X\rangle \ \text{and} \ \ann_R \left(\ext 2RMR\right)=\langle X, Y\rangle.$$

Thus, there is not a subset $\Sigma$ of $\ass_R(M)$ such that
$\ann_R\left(\ext 2RMR\right)=\ann_R\left(M/(\bigcap_{\fp_i\in\Sigma}M_i)\right)$.
 Moreover, for $t=2$,  this example shows that in the second inclusion of (\ref{ext1}) in Remark \ref{rem1}, to obtain  a better upper bound (under inclusion)  of $\ann_R\left(\ext tRMR\right)$, we can not replace the index set $\mass_R^{d-t}(M)$ by the larger  sets $\mass_R^{\geqslant d-t}(M)$, $\ass_R^{d-t}(M)$ or $\ass_R^{\geqslant d-t}(M)$.
\end{exa}

\begin{exa}\label{exa2}  Let $K$ be a field and  let $R=K[[X, Y, Z, W]]$ be the ring of formal power series over  $K$  in indeterminates $X, Y, Z, W$. Then $R$ is a local ring with maximal ideal $\fm=\langle X, Y, Z, W\rangle$. Set $\fp_1=\langle X, Y\rangle$, $\fp_2=\langle Z, W\rangle$, and $M=R/(\fp_1\cap\fp_2)$. Then $\depth_R{(R/\fp_1)}=\depth_R{(R/\fp_2)}=2$, and hence $\h i{\fm}{R/\fp_1}=\h i{\fm}{R/\fp_2}=0$ for $i=0, 1$. Now, the exact sequence $$0\rightarrow M\rightarrow {R}/{\fp_1}\oplus {R}/{\fp_2}\rightarrow {R}/{\fm}\rightarrow 0$$ induces the  exact sequence
\begin{align*}
0&\rightarrow \h 0{\fm}M\rightarrow \h 0{\fm}{{R}/{\fp_1}}\oplus \h 0{\fm}{{R}/{\fp_2}}\rightarrow \h 0{\fm}{{R}/{\fm}}
\rightarrow \h 1{\fm}M\\&\rightarrow \h 1{\fm}{{R}/{\fp_1}}\oplus \h 1{\fm}{{R}/{\fp_2}}
\end{align*}
of local cohomology modules. It follows that $\h 0{\fm}M=0$ and  $\h 1{\fm}M\cong R/\fm$. Since $R$ is a regular ring,
it is  Gorenstein \cite[Proposition 3.1.20]{bh},  and hence $R$ is the canonical module of $R$ \cite[Theorem 3.3.7]{bh}. Therefore, by the Grothendieck  duality \cite[Theorem 11.2.8]{bs}, we have $\Hom{R}{\ext 3RMR}{E(R/\fm)}\cong\h 1{\fm}M$.
 Thus $\ann_R\left(\ext 3RMR\right)=\fm$.

 On the other hand, if  $M_1=\fp_1/(\fp_1\cap\fp_2)$ and $M_2=\fp_2/(\fp_1\cap\fp_2)$, then $0=M_1\cap M_2$ is a minimal primary decomposition of the zero submodule of $M$. Since $\ass^1_R(M)=\emptyset$, we have $$R=\ann_R \left({M}/{\bigcap_{\fp_i\in\ass^1_R(M)}M_i}\right)\nsubseteq \ann_R\left(\ext 3RMR\right).$$

 Therefore in the first inclusion of (\ref{ext1}) in Remark \ref{rem1}, to obtain  a better lower bound  of $\ann_R\left(\ext tRMR\right)$,
   we can not replace the index set $\ass_R^{\geqslant d-t}(M)$ by the smaller set $\ass_R^{d-t}(M)$.
\end{exa}
\section{\bf Bounds for the annihilators of  local cohomology modules}
In this section we investigate the annihilators of local cohomology modules. For an $R$-module $M$, we denote $\sup\{i\in\N_0: \h i{\fa}M\neq 0\}$ by $\cd {\fa}M$. Let $\fa$ be a proper  ideal of $R$,  $M$  a non-zero finitely generated $R$-module of dimension $d$, and $0=M_1\cap\hdots \cap M_n$ a minimal primary
decomposition of the zero submodule of $M$ with $\ass_R(M/M_i)=\{\fp_i\}$ for all $1\leq i\leq n$.  If $\cd{\fa}M=d<\infty$, then
$$\ann_R\left(\h d{\fa}M\right)=\ann_R\left(M/\bigcap_{\cd{\fa}{R/\fp_i}=d}M_i\right);$$
see \cite{asn2} and Introduction  for more details.

For an arbitrary  integer $t$, when $(R, \fm)$ is a local ring, we give a bound for the $\ann_R\left(\h t{\fm}M\right)$, see Theorem \ref{annh}. Also,
 whenever $R$ is not necessarily local, in Theorem \ref{annh2}, we provide a bound for   $\ann_R\left(\h {\cd {\fa}M}{\fa}M\right)$  which implies the above equality when $\cd{\fa}M=\dim_R(M)$. Finally,  when $M$ is Cohen-Macaulay, a bound of $\ann_R\left(\h t{\fa}M\right)$ is given and at $t=\grad{\fa}M$ this annihilator is computed in Theorem \ref{annh3}.

Assume $(R, \fm)$ is  a local ring. The $\fm$-adic completion $\hat R$ of $R$ is a faithfully flat $R$-module (see \cite[Theorem 8.14]{mat}), and so $R\subseteq \hat R$. Applying \cite[Theorem 23.2]{mat} to the ring homomorphism $\varphi :R\rightarrow \hat R$  we obtain the following lemma.

\begin{lem}[See {\cite[Theorem 23.2]{mat}} ]\label{ass2} Let $(R, \fm)$ be a local ring, and $M$  an $R$-module.  Then
\begin{enumerate}[\rm(i)]
  \item If $\fp\in\spec(R)$ and $\fP\in\ass_{\hat R}(\hat R/\fp \hat R)$, then $R\cap\fP=\fp$.
  \item $$\ass_{\hat R}(M\otimes_{R}\hat R)=\bigcup_{\fp\in\ass_R(M)}\ass_{\hat R}(\hat R/\fp \hat R).$$
\end{enumerate}
\end{lem}

 This lemma is used in the proof of the following theorem which is the main theorem of this section.
 \begin{thm} \label{annh}Let $(R, \fm)$ be a  local ring  and $t\in\N_0$. Let $M$ be  a non-zero finitely generated $R$-module, and
  $0=M_1\cap\hdots\cap M_n$  a minimal primary decomposition of the zero submodule of $M$ with $\ass_R(M/M_i)=\{\fp_i\}$ for all $1\leq i\leq n$.
Then
\begin{enumerate}[\rm(i)]
\item $\bigcap_{\fp_i\in\ass_R^{\geqslant t}(M)}M_i=S^t_M(0)$ and $\bigcap_{\fp_i\in\mass_R^{t}(M)}M_i=T^t_M(0)$,  where $S^t=R\setminus\bigcup_{\fp\in\ass_R^{\geqslant t}(M)}\fp$ and $T^t=R\setminus\bigcup_{\fp\in\mass_R^t(M)}\fp$. In particular, $\bigcap_{\fp_i\in\ass_R^{\geqslant t}(M)}M_i$ and $\bigcap_{\fp_i\in\mass_R^{t}(M)}M_i$ are independent of the  choice of minimal primary decomposition of the zero submodule of $M$.
  \item    $S^t_M(0)$ is the largest  submodule $N$ of $M$ such that $\dim_R(N)<t$.
  \item $$ \ann_R\left({M}/{S^t_M(0)}\right)\subseteq\ann_R\left(\h t{\fm}M\right)
 \subseteq \ann_R\left({M}/{T^t_M(0)}\right).$$
 In particular, for $t=\dim_R(M)$, there are the equalities $S^t_M(0)=T^t_M(0)={\bigcap_{\fp_i\in\assh_R(M)}M_i}$, and
$$\ann_R\left(\h {\dim_R(M)}{\fm}M\right)=\ann_R\left({M}/{\bigcap_{\fp_i\in\assh_R(M)}M_i}\right).$$
\end{enumerate}
\end{thm}
\begin{proof}
Set $S=S^t_M(0)$ and $T=T^t_M(0)$.  It is clear that  $\ass_R^{\geqslant t}(M)$  and $\mass_R^t(M)$ are isolated subsets of $\ass_R(M)$, and hence (i) follows from Lemma \ref{primary}.

To prove (ii), first note that $\ass_R(S)=\ass_R(M)\setminus\ass_R^{\geqslant t}(M)$ by  Lemma \ref{ass} and hence $\dim_R(S)<t$.
Now, assume that $N$ is a submodule of $M$ such that $\dim_R(N)<t$. Suppose, for the sake of contradiction, that $N\nsubseteq S$. Then
 $$0\neq {N}/({N\cap S})\cong({N+S})/{S}\subseteq {M}/{S}.$$  Hence
  $$\emptyset\neq\ass_R\left( {N}/{\left(N\cap S\right)}\right)\subseteq\ass_R\left( {M}/{S}\right)=\ass_R^{\geqslant t}(M)$$
  which is impossible, because  $\dim_R \left({N}/\left({N\cap S}\right)\right)\leq\dim_R(N)<t$. This proves (ii).

Now, we prove (iii). In the case when $t=\dim_R(M)$, it is clear that
 $$\mass^t_R(M)=\ass^{\geqslant t}_R(M)=\assh_R(M),$$   and so  $S^t_M(0)=T^t_M(0)={\bigcap_{\fp_i\in\assh_R(M)}M_i}$. Therefore  the first part of (iii)  yields the equality
$\ann_R\left(\h {\dim_R(M)}{\fm}M\right)=\ann_R\left({M}/{\bigcap_{\fp_i\in\assh_R(M)}M_i}\right)$ whenever $t=\dim_R(M)$.
 Also, we saw in (ii) that $\dim_R(S)<t$, and so we obtain $\h t{\fm}M\cong\h t{\fm}{M/S}$. Therefore
 $$\ann_R\left({M/S}\right)\subseteq\ann_R\left(\h t{\fm}{M/S}\right)=\ann_R\left(\h t{\fm}M\right).$$
  Thus, to complete the proof of (iii), it only remains to prove that
   $$\ann_R\left(\h t{\fm}M\right)\subseteq\ann_R\left(M/T\right).$$
 Set $d=\dim_R(R)$. First, assume that $R$ is complete.   By the Cohen's structure theorem for complete local rings \cite[Theorem 29.4(ii)]{mat}, there is a  complete regular local ring $R'$  such that $R=R'/I$ for some ideal $I$ of $R'$. Now, let $h=\Ht {R'}I$ and  $x_1,\cdots,x_h$  a maximal $R'$-sequence  in $I$. Set ${R''}=R'/(x_1,\cdots,x_h)$ and
 $J=I/(x_1,\cdots,x_h)$. Then ${R''}$ is a  local Gorenstein ring of dimension $d$ \cite[Corollary 3.1.15]{bh} and $R\cong {R''}/J$.
 Now, let $\fn$ be the maximal ideal of ${R''}$. Then  $\fm\cong\fn/J$. By the  Grothendieck duality for Gorenstein rings \cite[Theorem 11.2.5]{bs}, there is the following isomorphism of $R''$-modules
 $$
 \h t{\fn}M\cong \Hom {R''}{\ext {d-t}{R''}M{R''}}{E_{R''}({R''}/\fn)}.
$$
Also, by using the Independence Theorem  under  the ring homomorphism  $R''\rightarrow R''/J\cong R$, we obtain
the following isomorphism of $R''$-modules
$$\h t{\fn}M\cong\h t{\fn (R''/J)}M\cong\h t{\fm}M$$
(we recall that $\fn/J\cong\fm$).
We refer the reader about the Independence Theorem  to \cite[Theorem 4.2.1]{bs} or \cite[Proposition 2.11(2)]{h}. Also, we note that any $R$-module $N$ has an $R''$-module structure given by $r''\cdot x=(r''+J)x=\psi(r''+J)x$ for all $r''\in R''$ and $x\in N$, where $\psi$ denotes the ring isomorphism from $R''/J$ to $R$. Hence, by \cite[Remarks 10.2.2(ii)]{bs}, we have
$$\ann_{R''}\left(\h t{\fm}M\right)=\ann_{R''}\left(\h t{\fn}M\right)=\ann_{R''}\left(\ext {d-t}{R''}M{R''}\right).$$

 Let, for each $1\leq i\leq n$, $P_i$ be the contraction of  $\fp_i$ in ${R''}$ under the ring homomorphism  $R''\rightarrow R''/J\cong R$. Then  $\ass_{R''}(M)=\{P_1,\ldots,P_n\}$ and there is the bijective  correspondence  between the sets $\ass_{R''}(M)$ and $\ass_R(M)$ given by $P_i\leftrightarrow\fp_i$. Also,
   $0=M_1\cap\hdots\cap M_n$ is  a minimal primary decomposition of the zero submodule of $M$ as  ${R''}$-modules with $\ass_{R''}(M/M_i)=\{P_i\}$ for all $1\leq i\leq n$. Since ${R''}$ is Gorenstein, by Remark \ref{rem1}(\ref{ext1}), we obtain \begin{align*}
  \ann_{R''}\left(\h t{\fm}M\right)\subseteq \ann_{R''}\left({M}/{\bigcap_{P_i\in\mass^{t}_{R''}(M)}M_i}\right).
  \end{align*}
 For any $R$-module $N$, we have $J\subseteq \ann_{R''}(N)$ and so $\ann_{{R''}/J}(N)=(\ann_{R''}(N))/J$.  Therefore the above inclusion  proves the  claimed inclusion  in the case where $R$ is complete.

 Now, suppose  that $R$ is not necessarily complete.    Assume $0=\bigcap_{k\in K}\mathcal{M}_k$ is a minimal $\hat R$-primary decomposition of the zero submodule of $\hat M$ with $\ass_{\hat R}(\hat M/\mathcal {M}_k)=\{\fP_k\}$. Since $\ass_{\hat R}(\hat M)=\bigcup_{i=1}^n\ass_{\hat R}(\hat R/\fp_i \hat R)$, there exists  subsets $K_1,\dots, K_n$ of $K$ such that $K=\bigcup_{i=1}^n {K_i}$,  and for each $i$, $\ass_{\hat R}(\hat R/\fp_i \hat R)=\{\fP_k: k\in K_i\}$. Also, the subsets $K_1,\dots, K_n$ of $K$ are disjoint by Lemma \ref{ass2}(i).

Assume  $x\in\ann_R\left(\h t{\fm}M\right)$ and $\fp_i\in\mass_R^{t}(M)$. By the complete case,
\begin{equation}\label{eqann}
x\hat R\subseteq\ann_{\hat R}\left(\h t{\fm\hat R}{\hat M}\right)\subseteq\ann_{\hat R}\left({\hat M}/{\bigcap_{\fP_k\in\mass_{\hat R}^{t}(\hat M)}\mathcal{M}_k}\right).
\end{equation}
Now, suppose that $k\in K_i$ and $\fP_k\in\assh_{\hat R}(\hat M/\widehat{M_i})$ (note that, by Lemma \ref{ass2}, $\ass_{\hat R}(\hat M/\widehat{M_i})=\ass_{\hat R}(\hat R/\fp_i\hat R)$). We have
 $$\dim_{\hat R}(\hat R/\fP_k)=\dim_{\hat R}(\hat M/\widehat{M_i})=\dim_R(M/M_i)=\dim_R(R/\fp_i)=t.$$
 We show that $\fP_k$ is a minimal element of $\ass_{\hat R}(\hat M)$.
  Assume that $1\leq i'\leq n$, $k'\in K_{i'}$,  and $\fP_{k'}\subseteq \fP_k$. Then $\fp_{i'}=\fP_{k'}\cap R\subseteq \fP_k\cap R=\fp_i$. Since $\fp_i$ is a minimal element of $\ass_R(M)$ and $K_1,\dots,K_n$ are disjoint sets, we deduce that $i=i'$. It follows that both  $\fP_k$ and $\fP_{k'}$ are elements of $\ass_{\hat R}(\hat M/\widehat{M_i})$. Therefore $$\dim_{\hat R}(\hat M/\widehat{M_i})=\dim_{\hat R}(\hat R/\fP_k)\leq\dim_{\hat R}(\hat R/\fP_{k'})\leq\dim_{\hat R}(\hat M/\widehat{M_i}),$$
and hence $\fP_k=\fP_{k'}$. Thus $\fP_k\in\mass_{\hat R}^{t}(\hat M)$ and  inclusion (\ref{eqann})   yields  $x\hat M\subseteq \mathcal{M}_k$. 
Since $\fP_k$ is a minimal element of  $\ass_{\hat R}(\hat M/\widehat{M_i})$, it follows that the contraction of $(\widehat{M_i})_{\fP_k}$  under the canonical map $\hat M\rightarrow \hat M_{\fP_k}$, denoted by $\mathcal N_k$, is the $\fP_k$-primary component of each minimal primary decomposition of $\widehat{M_i}$ in $\hat M$ (see Lemma \ref{primary} or \cite[Theorem 6.8.3(iii)]{mat}). Hence $\mathcal{N}_k/\widehat{M_i}$ is the $\fP_k$-primary component of each minimal primary decomposition of $0$ in $\hat M/\widehat{M_i}$. Also, we have   $\mathcal M_k\subseteq\mathcal N_k$ because $\mathcal M_k$ is the contraction of the zero submodule under the map $\hat M\rightarrow \hat M_{\fP_k}$. Therefore $x(\hat M/\widehat{M_i})\subseteq \mathcal{N}_k/\widehat{M_i}$.
 Since $\fP_k$ is an arbitrary element  of $\assh_{\hat R}(\hat M/\widehat{M_i})$, we have
 $x(\hat M/\widehat{M_i})\subseteq \bigcap_{\fP_k\in\assh_{\hat R}(\hat M/\widehat{M_i})}\mathcal{N}_k/\widehat{M_i}.$
  Hence, by Lemma \ref{ass},  $\ass_{\hat R}(x({\hat M}/{\widehat{M_i}}))\subseteq\ass_{\hat R}({\hat M}/{\widehat{M_i}})\setminus \assh_{\hat R}({\hat M}/{\widehat{M_i}})$. This yields
 $$\dim_R\left(x(M/M_i)\right)=\dim_{\hat R}(x({\hat M}/{\widehat{M_i}}))<\dim_{\hat R}(\hat M/\widehat{M_i})=t.$$ Therefore  $\fp_i\notin\ass_R(x(M/M_i))$ and hence $\ass_R(x(M/M_i))=\emptyset$ or equivalently $xM\subseteq M_i$. This proves the claimed inclusion and  completes the proof.
\end{proof}

    Now, in the following theorem, we give a bound for the annihilator of top local cohomology module  without the local assumption on $R$. But before that, we need the following lemma.
 \begin{lem}[{\cite[Theorem 2.2]{dnt}}]\label{cd}
 Let $\fa$ be an ideal of $R$ and  $M, N$ two finitely generated $R$-modules such that $\supp_R(M)\subseteq\supp_R(N)$. Then $\cd {\fa}M\leq\cd {\fa}N$.
\end{lem}
Assume $\fa$ is an ideal of $R$ and $M$ is a finitely generated $R$ module. Since $\supp_R(M)=\supp_R\left(\bigoplus_{\fp\in\ass_R(M)}R/\fp\right)$, above lemma implies that $$\cd{\fa}M=\operatorname{cd}_{R}\left(\fa, \bigoplus_{\fp\in\ass_R(M)}R/\fp\right)=\sup\{\cd{\fa}{R/\fp}: \fp\in\ass_R(M)\}.$$
By \cite[Exercise 6.2.6 and Theorem 6.2.7]{bs}, $\h i{\fa}M$ is non-zero for all $i$ if and only if $M=\fa M$,  and so,  in this case, we have $\cd{\fa}M=\sup\emptyset=-\infty$. On the other hand, if $\fa$ is generated by $t\in\N_0$ elements,  then $\cd{\fa}M\leq t<\infty$; see \cite[Theorem 3.3.1]{bs}. Hence $\cd{\fa}M$ is a non-negative integer if and only if $M\neq\fa M$.
\begin{thm} \label{annh2} Let $M$ be a non-zero finitely generated $R$-module and $\fa$ an ideal of $R$ such that $M\neq\fa M$. Let $c=\cd{\fa}M$
 and $0=M_1\cap\hdots\cap M_n$  a minimal primary decomposition of the zero submodule of $M$ with $\ass_R(M/M_i)=\{\fp_i\}$ for all $1\leq i\leq n$. Set $\Delta=\{\fp\in\ass_R(M): \cd{\fa}{R/\fp}=c\}$ and $\Sigma=\{\fp\in\ass_R(M): \cd{\fa}{R/\fp}=\dim_R (R/\fp)=c\}$.
Then
\begin{enumerate}[\rm(i)]
\item $\bigcap_{\fp_i\in\Delta}M_i=S_M(0)$,  where $S=R\setminus\bigcup_{\fp_i\in\Delta}\fp$. In particular, $\bigcap_{\fp_i\in\Delta}M_i$ is independent of the choice of minimal primary decomposition of the zero submodule of $M$.
  \item    $S_M(0)$ is the largest  submodule $N$ of $M$ such that $\cd{\fa}N<c$.
  \item $$\ann_R\left({M}/{\bigcap_{\fp_i\in\Delta}M_i}\right)\subseteq\ann_R(\h c{\fa}M)
 \subseteq \ann_R\left({M}/{\bigcap_{\fp_i\in\Sigma}M_i}\right).$$
In particular, when $c=\dim_R(M)$, there are the equalities $\Delta=\Sigma$ and
$$\ann_R\left(\h {\dim_R(M)}{\fa}M\right)=\ann_R\left({M}/{S_M(0)}\right).$$
\end{enumerate}
\end{thm}
\begin{proof} Set $S=\bigcap_{\fp_i\in\Delta}M_i$ and $T=\bigcap_{\fp_i\in\Sigma}M_i$.

(i) If $\fq\in\ass_R (M)$ and $\fq\subseteq \fp$ for some $\fp\in\Delta$, then, by Lemma \ref{cd},  $$c=\cd{\fa}{R/\fp}\leq\cd{\fa}{R/\fq}\leq\cd{\fa}M=c.$$
It follows that $\fq\in\Delta$, and hence $\Delta$ is an isolated subset of $\ass_R (M)$. Therefore  (i) follows from  Lemma \ref{primary}.

(ii) Lemma \ref{ass} implies that $\ass_R (S)=\{\fp\in\ass_R (M): \cd{\fa}{R/\fp}<c\}$. Hence, by Lemma \ref{cd},  $\cd {\fa}{S}<c$. Also, if $N$ is a submodule of $M$ such that $\cd{\fa}N<c$, then
$$\ass_R ({N}/({N\cap S}))=\ass_R (({N+S})/{S})\subseteq\ass_R ({M}/{S})=\Delta.$$
Thus, if $\ass_R ({N}/({N\cap S}))\neq\emptyset$, then $c=\cd{\fa}{{N}/({N\cap S})}\leq\cd{\fa}N$, which is impossible. Therefore
$N\subseteq S$ and the proof of (ii) is completed.

(iii) We proved  in (ii) that  $\cd{\fa}{S}<c$. Therefore $\h c{\fa}M\cong\h c{\fa}{M/S}$ and hence
 $$\ann_R(M/S)\subseteq\ann_R(\h c{\fa}{M/S})=\ann_R(\h c{\fa}M).$$
 This proves the first inclusion. Now, we prove the second inclusion claimed in (iii).

  {\bf Case 1:} Assume that $c=\dim_R(M)$ and $(R, \fm)$ is a complete local ring. For each prime ideal $\fp$, in view of the Grothendieck's Vanishing Theorem
  \cite[Theorem 6.1.2]{bs}, we have $\cd{\fa}{R/\fp}\leq\dim_R (R/\fp)$. It follows that $\Delta=\Sigma$, and so  $S=T$. Also, we have  $\Delta=\{\fp\in\assh_R(M): \sqrt{\fa+\fp}=\fm\}$ by the Lichtenbaum-Hartshorne Theorem. Therefore
  $$\sqrt{\fa+\ann_R(M/S)}=\sqrt{\fa+\bigcap_{\fp\in\ass_R (M/S)}\fp}=\sqrt{\fa+\bigcap_{\fp\in\Delta}\fp}.$$
 Since $M$ is a finitely generated $R$-module, the set $\Delta$ is  finite and so
 $$\sqrt{\fa+\bigcap_{\fp\in\Delta}\fp}=\sqrt{\bigcap_{\fp\in\Delta}(\fa+\fp)}=\bigcap_{\fp\in\Delta}\sqrt{\fa+\fp}=\fm.$$
 (Note that for ideals $\fa, \fb, \fc$ and  prime ideal $\fq$, we have $(\fa+\fb)\cap(\fa+\fc)\subseteq\fq$ if and only if $\fa+(\fb\cap\fc)\subseteq\fq$. Therefore
 $\sqrt{(\fa+\fb)\cap(\fa+\fc)} =\sqrt{\fa+(\fb\cap\fc)}$). Hence $\sqrt{\fa+\ann_R(M/S)}=\fm$ and  we deduce from the Independence Theorem
 $$\h c{\fa}M\cong\h c{\fa}{M/S}\cong\h c{\fm}{M/S}.$$
  Also, since $\ass_R(M/S)=\Delta=\Sigma\subseteq \assh_R(M)$ and $\Delta$ is not empty, we have $\dim_R(M/S)=\dim_R(M)=c$ and $\assh_R(M/S)=\ass_R(M/S)$. Therefore   the previous theorem yields
 $$\ann_R\left(\h c{\fa}M\right)=\ann_R\left(\h c{\fm}{M/S}\right)=\ann_R(M/S).$$

 {\bf Case 2:} Assume that $c=\dim_R(M)$ and $R$ is not necessarily local. As the before case, we have $\Delta=\Sigma$ and  $S=T$. To prove
  $\ann_R(\h c{\fa}M)\subseteq\ann_R(M/S)$, assume that $x\in R$ and
 $x M\nsubseteq S$ and we show  $x\h c{\fa}{M}\neq 0$. By (ii), $\h c{\fa}{xM}\neq 0$. Thus there exists a prime ideal $\fm$ such that $\h c{\fa R_\fm}{xM_\fm}\neq 0$ and consequently $\h c{\fa\widehat{R_\fm}}{x\widehat{M_\fm}}\neq 0$. Therefore $c\leq\operatorname {cd}_{\widehat{R_\fm}}({\fa\widehat{R_\fm}}, {x\widehat{M_\fm}})$. It follows from Lemma \ref{cd} and Grothendieck's Vanishing Theorem that
 $$c\leq \operatorname {cd}_{\widehat{R_\fm}}({\fa\widehat{R_\fm}}, {x\widehat{M_\fm}})\leq \operatorname {cd}_{\widehat{R_\fm}}({\fa\widehat{R_\fm}}, {\widehat{M_\fm}})\leq\dim_{\widehat{R_\fm}} (\widehat{M_\fm})\leq\dim_R(M)=c.$$
  Hence
 $\dim_{\widehat{R_\fm}} (\widehat{M_\fm})=\operatorname {cd}_{\widehat{R_\fm}}({\fa\widehat{R_\fm}}, {\widehat{M_\fm}})=c$. Since $\h c{\fa\widehat{R_\fm}}{x\widehat{M_\fm}}\neq 0$, we obtain $x{\widehat{M_\fm}}\nsubseteq S'$, where $S'$ is the largest submodule of ${\widehat{M_\fm}}$ such that $\operatorname {cd}_{\widehat{R_\fm}}({\fa\widehat{R_\fm}}, S')<c$. So, by the complete case,  we have $x\h c{\fa\widehat{R_\fm}}{\widehat{M_\fm} }\neq 0$ and therefore
  $x\h c{\fa}{M}\neq 0$. This proves the claimed inclusion (in fact equality) in the case where $c=\dim_R(M)$.

 {\bf Case 3:} Assume $c<\dim_R(M)$.  If $\Sigma=\emptyset$, then $T=M$ and there is nothing to prove. Assume $\Sigma\neq\emptyset$. Since $\cd{\fa}{T}\leq c$, the short exact sequence
   $$0\rightarrow T\rightarrow M\rightarrow M/T\rightarrow 0$$ induces the epimorphism
    $\h c{\fa}M\rightarrow\h c{\fa}{M/T}.$
  It follows that $\ann_R(\h c{\fa}M)\subseteq\ann_R(\h c{\fa}{M/T})$. Since $\ass_R (M/T)=\Sigma\neq\emptyset$, we have
    $$\cd{\fa}{M/T}=\max_{\fp\in\ass_R(M/T)}\cd{\fa}{R/\fp}=\max_{\fp\in\Sigma}\ \cd{\fa}{R/\fp}=c$$
    and
      $$\dim_R(M/T)=\max_{\fp\in\ass_R(M/T)}\dim_R(R/\fp)=\max_{\fp\in\Sigma}\ \dim_R(R/\fp)=c.$$
 Thus  $\dim_R (M/T)=\cd{\fa}{M/T}=c$, and so   $\ann_R(\h c{\fa}{M/T})=\ann_R({M/T})$  by  the previous case. This completes the proof.
\end{proof}

When $(R, \fm)$ is a Cohen-Macaulay local ring, $\fa$ is a non-zero proper ideal of $R$ and $t=\grad{\fa}R$, Bahmanpour calculated the annihilator   of $\h t{\fa}R$  in \cite[Theorem 2.2]{ba}. The following theorem generalizes his result for Cohen-Macaulay modules whenever $R$ is  not necessarily local.
\begin{lem}[{\cite[Theorem 2.1]{ctt}}] Let  $\fa$ be  an ideal of $R$, and $M$   a finitely generated  $R$-module such that $\fa M\neq M$. Then
$$\ass_R\left(\h {\grad{\fa}M}{\fa}M\right)=\{\fp\in V(\fa): \depth_{R_\fp}{M_\fp}=\grad{\fa}M\}. $$
\end{lem}

Let  $M$ be an $R$-module. For  $\fp\in\supp_R(M)$, the $M$-height of $\fp$, denoted $\Ht M{\fp}$, is the supremum  of the lengths $t$ of strictly descending   chains $$\fp=\fp_0\supset\fp_1\ldots\supset\fp_t$$
of prime ideals in $\supp_R(M)$. For an arbitrary ideal $\fa$, we define the $M$-height of $\fa$, denoted $\Ht M{\fa}$,   by
$$\Ht{M}{\fa}=\inf\{\Ht M{\fp}: \fp\in\supp_R(M)\cap \operatorname{V}(\fa)\}.$$
In particular, if $\supp_R(M)\cap \operatorname{V}(\fa)=\emptyset$, then $\Ht{M}{\fa}=\inf\emptyset=\infty$.

\begin{thm}\label{annh3} Let $\fa$ be an ideal of $R$, $M$  a non-zero finitely generated Cohen-Macaulay $R$-module, and
 $0=M_1\cap\hdots\cap M_n$ with $\ass_R(M/M_i)=\fp_i$ for all $1\leq i\leq n$ a minimal primary decomposition of the zero submodule  of $M$. Then, for each  $t\in\N_0$,
 $$\ann_R\left(\h t{\fa}M\right)\subseteq\ann_R\left({M}/{\bigcap_{\Ht{M}{\fa+\fp_i}=t } M_i}\right),$$
  Moreover,   if $M\neq\fa M $ and $t=\grad {\fa}M$, then equality holds.
\end{thm}
\begin{proof}
Set  $\Sigma(t)=\{\fp\in\ass_R(M): \Ht{M}{\fa+\fp}=t\}.$ To prove the claimed inclusion, assume that $x\in R$ and  $x\notin\ann_R({M}/{\bigcap_{\fp_i\in\Sigma(t)} M_i})$ and we show $x\notin\ann_R(\h t{\fa}M)$. Hence  $xM\nsubseteq M_i$ for some $\fp_i\in\Sigma(t)$. Therefore $\ass_R(x(M/M_i))=\ass_R(M/M_i)=\{\fp_i\}$. Suppose that $\fq$ is a minimal prime  ideal of $\fa+\fp_i$ such that $\Ht{M}{\fq}=\Ht{M}{\fa+\fp_i}=t$. Then
$$\ass_{R_\fq}(x(M/M_i)_\fq)=\ass_{R_\fq}(M/M_i)_\fq=\{\fp_iR_\fq\}.$$
Therefore
\begin{align*}
\sqrt{\fa R_\fq+\ann_{R_\fq}(M/M_i)_\fq}
=\sqrt{\fa R_\fq+\fp_i R_\fq}=\fq R_\fq.
\end{align*}
Also, since $M_\fq$ is Cohen-Macaulay and $\fp_i R_\fq\in\ass_{R_\fq}(M_\fq)$, we have $\dim_{R_\fq}(R_\fq/\fp_i R_\fq)=\dim_{R_\fq}(M_\fq)=t$. Hence
$\dim_{R_\fq}\left((M/M_i)_\fq\right)=t$
and, by Theorem \ref{annh} in view of the Independence Theorem, we have
$$\ann_{R_\fq}\left(\h t{\fa R_\fq}{(M/M_i)_\fq}\right)=\ann_{R_\fq}\left(\h t{\fq R_\fq}{(M/M_i)_\fq}\right)=\ann_{R_\fq}(M/M_i)_\fq.$$
 Thus $x\h t{\fa R_\fq}{(M/M_i)_\fq}\neq 0$ because $x(M/M_i)_\fq\neq 0$.
On the other hand, the exact sequence $$0\rightarrow (M_i)_\fq\rightarrow M_\fq\rightarrow (M/M_i)_\fq\rightarrow 0$$
induces the epimorphism $\h t{\fa R_\fq}{M_\fq}\rightarrow\h t{\fa R_\fq}{(M/M_i)_\fq}.$ Thus $x\h t{\fa R_\fq}{M_\fq}\neq 0$ and consequently $x \h t{\fa}M\neq 0$. This proves the claimed inclusion.

 Finally, assume $t=\grad{\fa}M$ and we prove the reverse inclusion. Let $x\in R$ be such that $x\h t{\fa}M\neq 0$. Hence, there exists $\fq\in\ass_R(\h t{\fa}M)\subseteq \supp_R(M/\fa M)$ such that $x\h t{\fa R_\fq}{M_\fq}\neq 0$. By above lemma, $\Ht {M}\fq=\Ht M\fa$, and hence $\fq$ is a minimal prime ideal of $\fa+\ann_R(M)$. Since $M_\fq$ is a Cohen-Macaulay module of dimension $t$,   Theorem \ref{annh} and  Independence Theorem yield
 $$\ann_{R_\fq}\left(\h t{\fa R_\fq}{M_\fq}\right)=\ann_{R_\fq}\left(\h t{\fq R_\fq}{M_\fq}\right)=\ann_{R_\fq}(M_\fq),$$
  and so we have $xM_\fq\neq 0$.
 If  $\fq\in\supp_R\left({\bigcap_{\fp_i\in\Sigma(t)} M_i}\right)$, then there is a $\fp\in\ass_R\left({\bigcap_{\fp_i\in\Sigma(t)} M_i}\right)=\ass_R(M)\setminus \Sigma(t)$ such that $\fp\subseteq\fq$. Therefore  $$t=\Ht M\fa\leq\Ht M{\fa+\fp}\leq\Ht M\fq=t.$$
 Hence $\Ht M{\fa+\fp}=t$, and so $\fp\in\Sigma(t)$,  a contradiction. Thus $\left({\bigcap_{\fp_i\in\Sigma(t)} M_i}\right)_\fq=0$. It follows that $xM_\fq\nsubseteq \left({\bigcap_{\fp_i\in\Sigma(t)} M_i}\right)_\fq$ and consequently $xM\nsubseteq {\bigcap_{\fp_i\in\Sigma(t)} M_i}$. This proves the claimed equality in the case where $t=\grad {\fa}M$ and completes the proof.
\end{proof}

\section*{Acknowledgements}
The author is deeply grateful to the referee for a very
careful reading of the manuscript and many valuable suggestions in improving the quality of the paper. Also, the author would like to thank Professor Hossein Zakeri for his careful reading of the first draft and many helpful suggestions.

  \end{document}